\documentclass{fundam}

\publyear{26}
\papernumber{1}
\volume{196}
\issue{1}
\theDOI{10.46298/fi.14941}

\versionForARXIV

\usepackage{graphicx}
\usepackage{times}
\usepackage{mathptmx}

\usepackage{cite}
\usepackage[T1,OT1]{fontenc}
\usepackage{textcomp}
\usepackage{xcolor}

\usepackage[ruled,linesnumbered]{algorithm2e}
\usepackage{multirow}
\usepackage{mathrsfs,amssymb}
\usepackage{stmaryrd,amsmath,latexsym,indentfirst}
\usepackage{stmaryrd}
\usepackage{makecell}
\usepackage{booktabs}
\usepackage{xcolor}
\usepackage{subfig}
\usepackage{bm}
\usepackage[symbol]{footmisc}
\usepackage{cellspace} 

\usepackage[capitalise]{cleveref} 

\newtheorem{mytheorem}{Theorem}[section]
\newtheorem{mylemma}[mytheorem]{Lemma}
\newtheorem{mycorollary}[mytheorem]{Corollary}

\begin{document}

\title{{Structure Fault Diameter of Hypercubes}\thanks{The research was funded by Natural Science Foundation of Xinjiang Uygur Autonomous Region (No. 2025D01C34) and Natural Science Foundation of China (No. 12261085)}}

\author{Honggang Zhao\\
College of Mathematics and System Sciences, Xinjiang University \\
Urumqi, 830046, P. R. China \\
\and
Eminjan Sabir \\
College of Mathematics and System Sciences, Xinjiang University \\
Urumqi, 830046, P. R. China \\
\and
Cheng-Kuan Lin \\
Department of Computer Science, National Yang Ming Chiao Tung University \\
Hsinchu 30010, Taiwan}

\maketitle

\runninghead{H. Zhao, E. Sabir, C.-K. Lin}{Structure Fault Diameter of Hypercubes}

\vspace{-7ex}

\begin{abstract}
Structure connectivity and substructure connectivity are innovative indicators for assessing network reliability and fault tolerance. Similarly, fault diameter evaluates fault tolerance and transmission delays in networks. This paper extends the concept of fault diameter by introducing two new variants: structure fault diameter and substructure fault diameter, derived from structure connectivity and substructure connectivity respectively.
For a connected graph $G$ with $W$-structure connectivity $\kappa(G;W)$ or $W$-substructure connectivity $\kappa^s(G;W)$, the $W$-structure fault diameter $D_f(G;W)$ and $W$-substructure fault diameter
$D_f^s(G;W)$
 are defined as the maximum diameter of any subgraph of $G$ resulting from removing up to $\kappa(G;W)-1$ $W$-structures or $\kappa^s(G;W)-1$ $W$-substructures. For the
$n$-dimensional hypercube $Q_n$ with $n \geq 3$ and
$1 \leq m \leq n - 2$, we determine both $D_f(Q_n;Q_m)$ and $D_f^s(Q_n;Q_1)$. These findings generalize existing results for the diameter and fault diameter of $Q_n$, providing a broader understanding of the hypercube's structural properties under fault conditions.
\end{abstract}



\section{Introduction}

In the study of communication networks, graphs serve as powerful tools for modeling network structures and analyzing their properties.  The \textit{connectivity} and \textit{diameter} are fundamental parameters to measure fault tolerance and communication delay.
A reliable communication network must not only withstand faults but also maintain a minimal diameter to ensure efficient communication despite failures. This is particularly crucial in large-scale distributed systems, where disruptions can severely affect performance. To tackle this issue, the concept of \textit{fault diameter} has been introduced, which evaluates the impact of faults on a network's diameter.

The \textit{fault diameter}, $D_f(G)$, is defined as the maximum diameter of any subgraph of a connected graph $G$ obtained after removing up to
$\kappa(G)-1$ vertices, where $\kappa(G)$ represents  the graph's connectivity. The study of fault diameter provides critical insights into a network's resilience to failures and the impact of faults on communication delay. This is particularly relevant in applications such as data centers, cloud computing, and parallel processing, where maintaining low-latency communication is essential.
Analyzing fault diameter deepens our understanding of graph structures and their robustness under adversarial
conditions. This analysis provides valuable insights for designing resilient network topologies capable of effectively managing node failures. For example, hypercube networks and their variations are extensively employed in distributed computing due to their exceptional characteristics, such as symmetry, scalability, and inherent fault tolerance.  A thorough understanding of their fault diameters is essential for optimizing these networks to maintain performance and reliability during failure scenarios.

Krishnamoorthy and Krishnamurthy first introduced the concept of fault diameter, demonstrating that the fault diameter of the $n$-dimensional hypercube $Q_n$ is $n + 1$ \cite{03}. This foundational work has since been expanded to more intricate network structures. Tsai et al. studied the exchanged hypercube $EH(s, t)$ and discovered that after removing at most $s$ vertices, the diameter of the resulting graph is $s + t + 3$ for $3 \leq s \leq t$ \cite{08}. Qi and Zhu established upper bounds for the fault diameters of two families of twisted hypercubes, $H_n$ and $Z_{n, k}$ \cite{09}. Additionally, Day and Al-Ayyoub found that the fault diameter of the $k$-ary $n$-cube $Q_n^k$ increases by at most one compared to its fault-free diameter \cite{13}. Similar findings have been reported for other topologies, including star graphs \cite{15}, hierarchical cubic networks \cite{17} and exchanged crossed cubes \cite{12}.

Despite these advancements, there remains a need to investigate fault diameters across a wider range of graph structures, particularly within modern network models that incorporate complex and hierarchical designs. Such research not only enriches the theoretical understanding of network robustness but also provides practical insights for designing reliable and efficient communication systems in environments prone to faults. This paper aims to address this gap by introducing new fault diameter concepts based on structure connectivity and substructure connectivity, and applying these concepts to analyze the fault-tolerant properties of $Q_n$ under various fault conditions.

By considering the impact of structures becoming faulty instead of individual vertices, Lin et al. introduced the notions of structure connectivity and substructure connectivity \cite{02}. For a connected graph $G$, let $W$ be a subgraph of $G$. Then $W$-\textit{structure connectivity} (resp., $W$-\textit{substructure connectivity}) of $G$, denoted $\kappa(G;W)$ (resp., $\kappa^s(G;W)$), is the cardinality of a minimal set of vertex-disjoint subgraphs $\mathcal{W} = \{W_1, W_2, \ldots, W_t\}$, such that each $W_k \in \mathcal{W}$ is isomorphic to $W$ (resp., each $W_k \in \mathcal{W}$ is a connected subgraph of $W$) for $k = 1, 2, \ldots, t$, and removing $\mathcal{W}$ disconnects $G$. They also determined $\kappa(Q_n; W)$ and $\kappa^s(Q_n; W)$ for the structure $W \in \{K_1, K_{1,1}, K_{1,2}, K_{1,3}, C_4\}$.

Following this trend, many scholars have engaged in this research field. For instance, in the split-star networks $S^2_n$, Zhao and Wang determined both $\kappa(S^2_n; W)$ and $\kappa^s(S^2_n; W)$ for $W \in \{P_t, C_q\}$, where $4 \le t \le 3n - 5$ and $6 \le q \le 3n - 5$ \cite{22}. Ba et al.~investigated $P_t$-structure connectivity and $P_t$-substructure connectivity of augmented $k$-ary $n$-cubes $AQ^k_n$ \cite{23}. Yang et al. proved that $\kappa(S_n; K_{1,m}) = \kappa^s(S_n; K_{1,m}) = n - 1$ for $n \ge 4$ and $0 \le m \le n - 1$, where $S_n$ is a star graph \cite{24}. Recently, Guo et al. studied the $K_{1,r}$-structure connectivity of folded crossed cubes and hierarchical folded cubic networks \cite{33, 34}. Wang et al. proposed the concept of \textit{double-structure connectivity} and studied the double-structure connectivity of hypercubes \cite{21}. For the $n$-dimensional hypercube $Q_n$, Sabir and Meng considered a special kind of substructure connectivity, called \textit{$W$-subcube connectivity} $\kappa^{sc}(Q_n; W)$, by restricting the structure $W$ and its subgraphs to subcubes of $Q_n$ \cite{04}.

In this paper, we propose two novel extensions of the fault diameter, defined based on the concepts of structure connectivity and substructure connectivity. The $W$-\textit{structure fault diameter}, denoted as $D_f(G;W)$, of a connected graph $G$ with $W$-structure connectivity $\kappa(G;W)$, is the maximum diameter of any subgraph of $G$ obtained by removing up to $\kappa(G;W) - 1$ $W$-structures. Similarly, the $W$-\textit{substructure fault diameter}, denoted as $D^s_f(G;W)$, of $G$ with $W$-substructure connectivity $\kappa^s(G;W)$, is the maximum diameter of any subgraph of $G$ obtained by removing up to $\kappa^s(G;W) - 1$ $W$-substructures. Importantly, when $W$ is a single vertex (i.e., $K_1$), the $W$-structure fault diameter and $W$-substructure fault diameter reduce to the traditional fault diameter. Note that the $W$-structure connectivity is defined based on $W$-structure sets, i.e., a $W$-structure set is a set of graphs whose every element is isomorphic to $W$. Similarly, the $W$-substructure connectivity is defined based on $W$-substructure sets, i.e., a $W$-substructure set is a set of graphs whose every element is isomorphic to a connected subgraph of $W$. This means that every element in a $W$-substructure is allowed to be isomorphic to $W$. So, by definition of $W$-structure fault diameter and $W$-substructure fault diameter, one can infer that $D_f(G;W)\le D^s_f(G;W)$ when $\kappa(G;W)=\kappa^s(G;W)$. For instance, for the structure $K_{1,2}$ and a graph $G$, we set $\mathcal{W}_1$ as a $K_{1,2}$-structure set and $\mathcal{W}_2$ as a $K_{1,2}$-substructure set. This implies that every element in $\mathcal{W}_1$ is $K_{1,2}$, yet elements in $\mathcal{W}_2$ may be $K_1$, $K_{1,1}$, or $ K_{1,2}$. So, when $\kappa(G;W)=\kappa^s(G;W)$, the choice range of structures in $\mathcal{W}_2$ is broader than $\mathcal{W}_1$. This infer that $D_f(G;W)$ dose not exceed $D^{s}_f(G;W)$ if $\kappa(G;W)=\kappa^s(G;W)$.

The $n$-dimensional hypercube $Q_n$, known for its symmetry, scalability, and fault tolerance, is one of the most popular interconnection networks. For more relevant research, we can refer to \cite{04,05,06,07,18,31,32}. It is well established that the diameter $D(Q_n)$ and the fault diameter $D_f(Q_n)$ of $Q_n$ are $n$ and $n + 1$, respectively. In this paper, we extend these results by proving the following:

\begin{enumerate}
    \item $D_f(Q_n;Q_m) = n$ for $n = m + 2$ and $D_f(Q_n;Q_m) = n + 1$ for $n \geq m + 3$.
    \item $D^s_f(Q_n;Q_1) = n + 1$ for $n \geq 4$, where $Q_1 \cong K_2$.
\end{enumerate}

The rest of this paper is organized as follows. In Section 2, we introduce the definitions and notations used throughout this study. In Section 3, we present our main results and proofs. Finally, in Section 4, we conclude the paper and discuss potential directions for future research.

\section{Preliminaries}

The definitions and notation of graph are based on \cite{01}. Let $G=(V,E)$ be a \textit{graph} with vertex set $V(G)$ and edge set $E(G)$. We use a vertex $v\in G$ and an edge $e\in G$ to represent $v\in V(G)$ and $e\in E(G)$, respectively. A graph $G$ is \textit{vertex transitive} if there is an isomorphism $f$ from $G$ into itself such that $f(u)=v$ for any two vertices $u$ and $v$ of $G$. A graph $G$ is \textit{edge transitive} if there is an isomorphism $f$ from $G$ into itself such that $f((u,v))=(x,y)$ for any two edges $(u,v)$ and $(x,y)$. For a vertex $u$ in a graph $G$, $N_G(u)$ denotes the \textit{neighborhood} of $u$, which is the set $\{v \mid (u,v)\in E\}$. A \textit{path} $P$ is a sequence of adjacent vertices, written as $\langle u_1, u_2, \ldots, u_n \rangle$. The \textit{length} of a path $P$, denoted $l(\textit{P})$, is the number of edges in $P$. We also write the path $\langle u_1, u_2,\ldots, u_n \rangle$ as $\langle u_1, P_1, u_i, u_{i+1},\ldots, u_j, P_2, u_t,\ldots, u_n \rangle$, where $P_1$ is the path $\langle u_1, u_2,\ldots, u_i \rangle$ and $P_2$ is the path $\langle u_j, u_{j+1},\ldots, u_t \rangle$. Hence, it is possible to write a path as $\langle u_1, Q, u_1, u_2,\ldots, u_n \rangle$ if $l(Q)=0$. We use $d_G(u,v)$ to denote the \textit{distance} between $u$ and $v$, that is, the length of a shortest path joining $u$ and $v$ in $G$. The $diameter$ of a graph $G$, denoted $D(\textit{G})$, is defined as max$\{d(u,v) \mid u,v \in V(G)\}$. We use $\langle u, P_s, v \rangle$ to denote the shortest path between $u$ and $v$ in a graph $G$. And we use $K_n$ to represent the complete graph with $n$ vertices.

\begin{figure}
	\centering
	\includegraphics[width=0.6\linewidth]{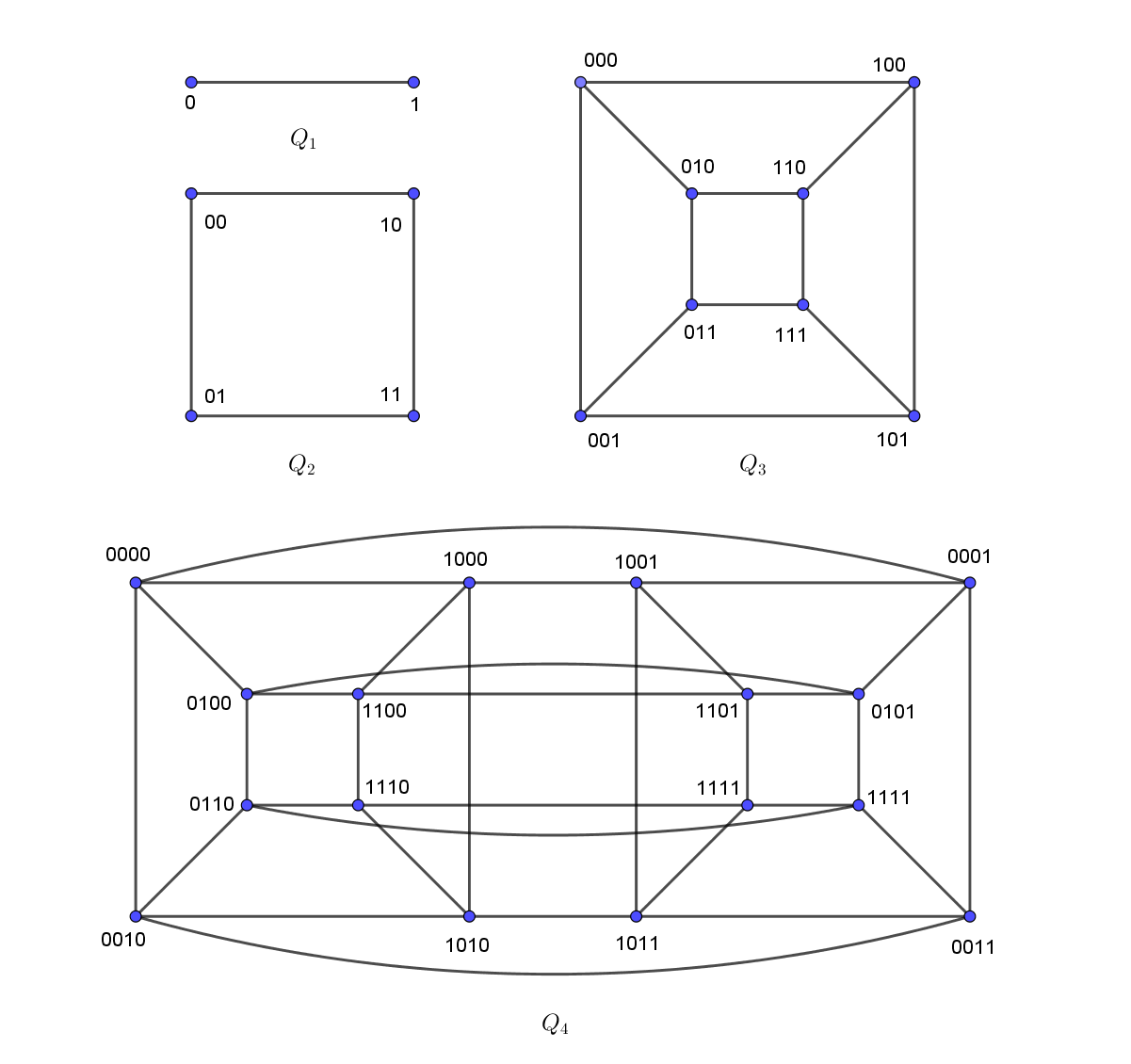}
	\caption{The $n$-dimensional hypercube for $n\in\{1,2,3,4\}$.}
	\label{fig:1}
\end{figure}

An $n$-\textit{dimensional hypercube} is an undirected graph, $Q_n$, with $2^n$ vertices and $2^{n-1}n$ edges. Each vertex in $Q_n$ can be represented as an $n$-bit binary string. We use boldface to denote vertices in $Q_n$. For any vertex $\textbf{x}={x_1}{x_2}\cdots{x_n}$ in $Q_n$, we set $(\textbf{x})^i={x^i_1}{x^i_2}\cdots{x^i_n}$ is the neighbor of $\textbf{x}$ in dimension $i$, where $x^i_j=x_j$ for every $j \ne i$ and $x^i_i=1-x_i$. In particular, $Q_0$ represents $K_1$ and $Q_1$ represents $K_2$. The $x_i$ in $\textbf{x}={x_1}{x_2}\cdots{x_n}$ is defined as $i$th bit. Fig.~\ref{fig:1} shows $Q_n$ for  $n\in\{1,2,3,4\}.$ By fixing the $n$th bit of the vertices in $Q_n$, we get two $(n-1)$-dimensional hypercubes named of ${Q^0_n}$ whose $n$th bit is $0$ and ${Q^1_n}$ whose $n$th bit is $1$, respectively. In this way, we divide $Q_n$ into two parts ${Q^0_n}$ and ${Q^1_n}$. For any vertex $\textbf{x}$ in ${Q^0_n}$ (resp., in ${Q^1_n}$), there exists an unique external neighbor $(\textbf{x})^n$ in ${Q^1_n}$ (resp., in ${Q^0_n}$). It is known that $Q_n$ has many attractive properties, such as being bipartite, $n$-regular, $n$-connected, vertex transitive and edge transitive \cite{18}.

For the $Q_m$-subcube connectivity $\kappa^{sc}(Q_n;Q_m)$ of $Q_n$, we define the $Q_m$-\textit{subcube fault diameter} $D^{sc}_f(Q_n;Q_m)$ is the maximum diameter of any sub-graph of $Q_n$ obtained by removing $\kappa^{sc}(Q_n;Q_m)-1$ or less $Q_m$-subcubes. It is worth noting that every $Q_m$-subcube is isomorphic to some $Q_i$, $0\le i\le m$, which is not the same as $Q_m$-substructure if $m\ge 2$. In particular, $Q_m$-subcube fault diameter is the same as $Q_m$-substructure fault diameter when $m=0,1$.

\begin{figure}
	\centering
	\includegraphics[height=6cm]{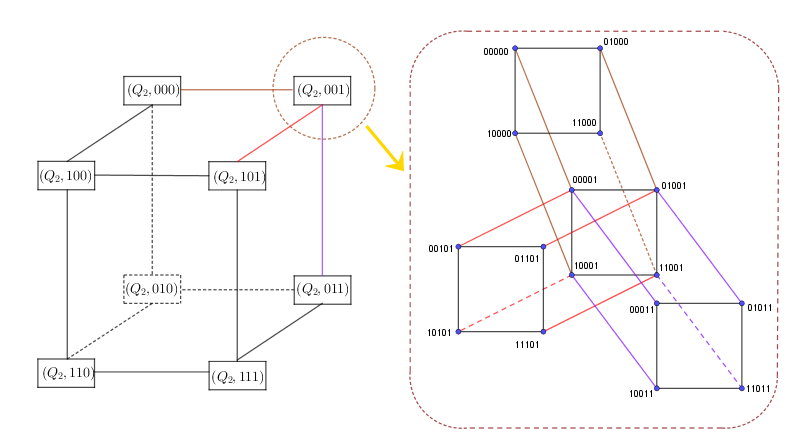}
	\caption[Fig.2]{$Q_5=Q_2\Box Q_3$.}
	\label{fig:2}
\end{figure}

\begin{figure}
	\centering
	\includegraphics[height=5cm]{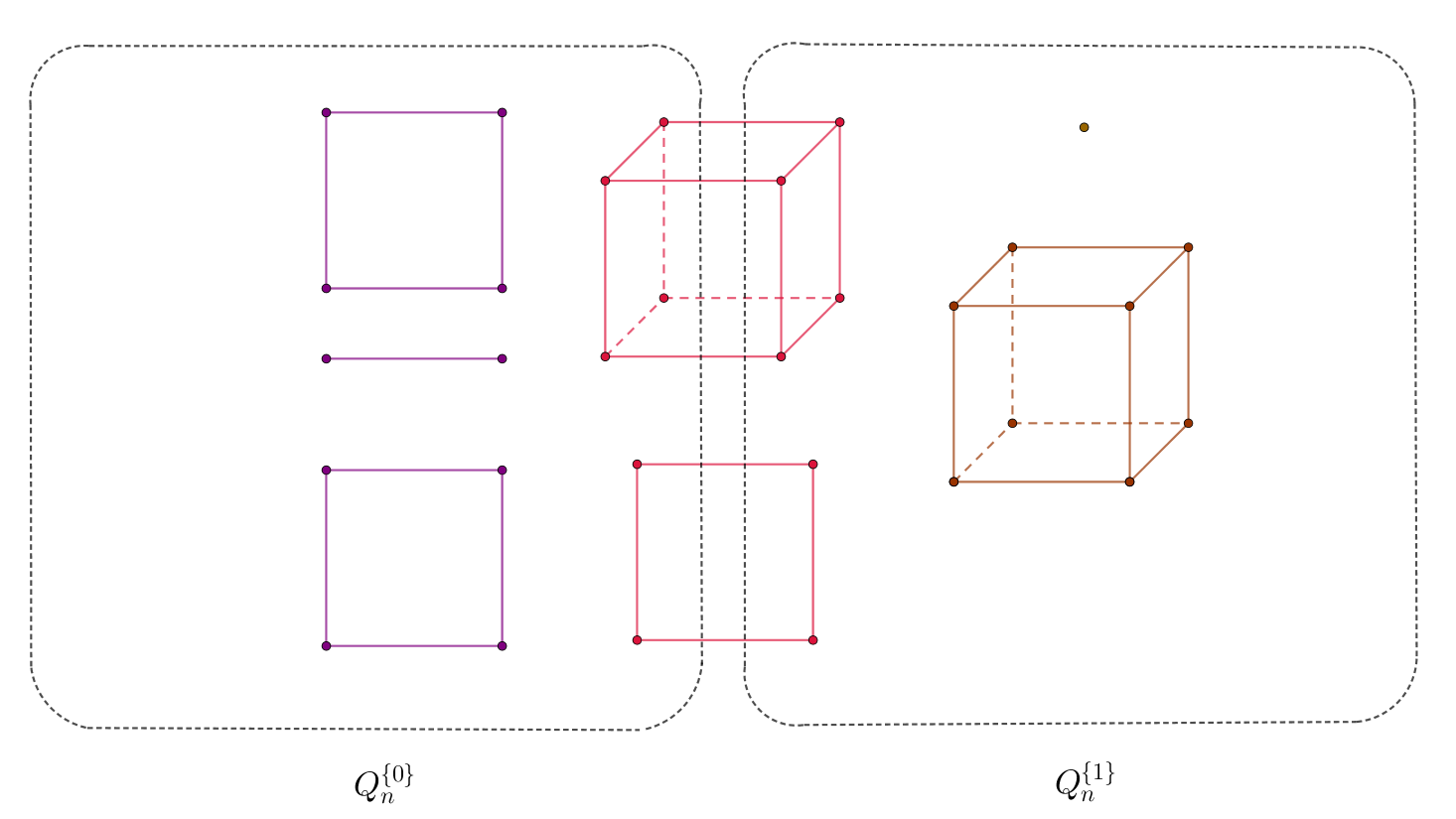}
	\caption[Fig.3]{An example of $|\mathcal{F}_3^n|=7$, $|\mathcal{A}^n_{3,0}|=3$, $|\mathcal{A}^n_{3,1}|=2$,
			$|\mathcal{B}^n_3|=2$.
			$|F^n_3|=2$, $| A^n_{3,0}|=0$, $|A^n_{3,1}| =1$ and $|B^n_3|=1$.}
	\label{fig:3}
\end{figure}

The \textit{cartesian product} of simple graphs $G$ and $H$ is the graph $G\Box H$ whose vertex set is $V(G)\times V(H)$ and whose edge set is the set of all pairs $(u_1v_1,u_2v_2)$ such that either $(u_1,u_2)\in E(G)$ and $v_1=v_2$, or $(v_1,v_2)\in E(H)$ and $u_1=u_2$ \cite{01}. Hypercubes can also be represented in the form of cartesian product, i.e., $Q_n=\underbrace{K_2 \Box K_2 \Box \cdots \Box K_2}_n$ \cite{14}. In this way, we can decompose $Q_n=Q_m\Box Q_{n-m}$. Now, for any $\textbf{t}\in V(Q_{n-m})$ we denote by $(Q_m,\textbf{t})$ the subgraph of $Q_n$ induced by the vertices whose last $n-m$ bits form the tuple $\textbf{t}$. We say that $(Q_m,\textbf{t})$ is adjacent to $(Q_m,\textbf{s})$ if $\textbf{t}$ is adjacent to $\textbf{s}$. It is easy to observe that $(Q_m,\textbf{t})$ is isomorphic to $Q_m$. As $Q_{n-m}$ is $(n-m)$-regular and $(n-m)$-connected, every vertex in $V(Q_{n-m})$ is adjacent to exactly $n-m$ vertices in $Q_{n-m}$. Let $N_{Q_{n-m}}(\textbf{t})=\{\textbf{t}_1, \textbf{t}_2,\ldots, \textbf{t}_{n-m}\}$.  Hence induced subgraph $(Q_m,\textbf{t})$ of $Q_n$ is adjacent to exactly $n-m$ subcubes, namely $(Q_m,\textbf{t}_1)$, $(Q_m,\textbf{t}_2)$,$\ldots, (Q_m,\textbf{t}_{n-m})$. Clearly, $(Q_m,\textbf{t}_i)$ is not adjacent to $(Q_m,\textbf{t}_j)$ for $1\le i,j\le n-m$, and $(Q_m,\textbf{t})$ and $(Q_m,\textbf{t}_i)$ can form a subcube, namely $(Q_m,\textbf{t}^*_i)$, which is isomorphic to $Q_{m+1}$. Fig.~\ref{fig:2} shows $Q_5=Q_2\Box Q_3$.

For any two vertices $\textbf{u}$, $\textbf{v}\in Q_n$, the \textit{Hamming distance} $H_{Q_n}(\textbf{u}$, $\textbf{v})$ is defined to be the number of different positions between the two strings. Then $\textbf{u}$ and $\textbf{v}$ are called \textit{symmetric} if $H_{Q_n}(\textbf{u}$, $\textbf{v})=n$, and $\textbf{u}$ and $\textbf{v}$ are called \textit{unsymmetric} if $H_{Q_n}(\textbf{u}$, $\textbf{v})\le n-1$. By definition of hypercubes, we know that any pair of vertices is either symmetric or unsymmetric in $Q_n$. We list some symbols in Table 1 and their illustrations in Fig.~\ref{fig:3}. It is worth noting that in Fig.~\ref{fig:3}, the faulty structures that fall completely within $Q^0_n$ are marked purple, and the faulty structures that fall completely within $Q^1_n$ are marked brown, and the faulty structures that half fall within $Q^0_n$ and half fall within $Q^1_n$ are marked red.

\begin{table}
	\label{Table11}
	\caption{Symbol table}
	\centering
	\footnotesize
	\begin{tabular}{ll}
		\toprule
		{\bf Symbol} & {\bf Definition}\\
		\midrule
		$\kappa(G;W)$ & $W$-structure connectivity of $G$\\
		$\kappa^s(G;W)$ & $W$-substructure connectivity of $G$\\
		$D_f(G;W)$ & $W$-structure fault diameter of $G$\\
		$D^s_f(G;W)$ & $W$-substructure fault diameter of $G$\\
		$Q_n$ & the $n$-dimensional hypercube\\
		$\kappa^{sc}(Q_n;Q_m)$ & $Q_m$-subcube connectivity of $Q_n$\\
		$D^{sc}_f(Q_n;Q_m)$ & $Q_m$-subcube fault diameter of $Q_n$\\
		${Q^h_n}$ & the $(n-1)$-dimensional hypercube with $V({Q^h_n})=\{\textbf{x}\mid\textbf{x}={x_1}{x_2}\cdots{x_n}$, $x_n=h\}$,\\
		& where $h\in \{{0,1}\}$\\
		$S_k(Q_n)$ &  the set $\{ U \mid U \subseteq V(Q_n)$ and the subgraph induced by $U$ is isomorphic to $Q_k \}$\\
		$\mathcal{F}_k^n$ & the vertex-disjoint subset of $\bigcup^k_{i=0} S_i(Q_n)$, i.e., any two distinct $A, B \in \mathcal{F}_k^n$\\
		& have no common vertex\\
		$\mathcal{A}^n_{k,h}$ & the set of $\mathcal{F}^n_k\cap \bigcup^k_{i=0}S_i({Q^h_n})$, i.e., for any $A \in \mathcal{A}^n_{k,h}$, we have $A\subseteq V(Q^h_n)$\\
		$\mathcal{B}^n_k$ & the set of $\mathcal{F}^n_k\setminus (\mathcal{A}^n_{k,0}\cup \mathcal{A}^n_{k,1})$, i.e., for any $B \in \mathcal{B}^n_{k}$, exactly half of the vertices in $B$ \\
		& fall within $Q^0_n$, and the other half fall within $Q^1_n$\\
		$F_k^n$ & the subset of $\mathcal{F}^n_k$, and for any $A \in F_k^n$, we have $A\in S_k(Q_n)$\\
		$A^n_{k,h}$ & the set of $F^n_k\cap S_k({Q^h_n})$\\
		$B^n_k$ & the set of $F^n_k\setminus (A^n_{k,0}\cup A^n_{k,1})$\\
		$E^n$ & the set of edges which connect ${Q^0_n}$ and ${Q^1_n}$\\
		\bottomrule
	\end{tabular}
\end{table}

The following results play crucial role in the proof of our main results.

\begin{mylemma}\label{lemma3.2}\cite{07}
	For $n\ge 2$, after the removal of $n-2$ or less vertices in $Q_n$, the diameter of the remaining graph is still $n$.
\end{mylemma}

\begin{mylemma}\label{lemma2.2} \cite{03}
For $n\ge 3$, $D_f(Q_n)=n+1$.
\end{mylemma}

\begin{mylemma}\label{lemma2.3} \cite{02} For $n\ge 3$,
	$\kappa(Q_n;Q_1)=\kappa^s(Q_n;Q_1)=n-1.$
\end{mylemma}

\begin{mylemma}\label{lemma2.4} \cite{04} For $n\ge 3$ and $m\le n-2$,
	$\kappa^{sc}(Q_n;Q_m) = \kappa(Q_n;Q_m) = n-m$.
\end{mylemma}

\begin{mylemma}\label{lemma2.5} \cite{06} Any two vertices $\textbf{u}$ and $\textbf{v}$ in $Q_n(n\ge 3)$ have exactly $2$ common neighbors if they have any. Besides, there are two common neighbors if and only if $((\textbf{u})^i)^j=\textbf{v}$, where $1\le i\ne j\le n$.
\end{mylemma}

Let $Q_m$ be a subcube of $Q_n$. For any two vertices $\textbf{u}$ and $\textbf{v}$ in $Q_m(m\ge 2)$, if $\textbf{u}$ and $\textbf{v}$ have common neighbors, then by Lemma~\ref{lemma2.5}, they have exactly two common neighbors and $H_{Q_n}(\textbf{u},\textbf{v})=H_{Q_m}(\textbf{u},\textbf{v})=2$. Clearly, their common neighbors are in $Q_m$. Moreover, the two vertices of $Q_1$ have no common neighbors. Then we have the following corollary of Lemma~\ref{lemma2.5}.

\begin{mycorollary}\label{corollary2.6} Let $Q_m$ be a subcube of $Q_n$. Then, any two vertices of $Q_m$ have no common neighbor in $Q_n-Q_m$.
\end{mycorollary}

We get the following lemma easily by the cardinality of symmetric vertices.

\begin{mylemma}\label{lemma2.7} For $n\ge 2$, let $S$ be any vertex set of $Q_n$ with $| S|< 2^{n-1}$. If $Q_n-S$ is connected, then $D(Q_n-S)\ge n$.
\end{mylemma}

\section{$Q_1$-structure fault diameter and $Q_1$-substructure fault diameter}

Lin et al. proved that  $\kappa(Q_n;Q_1)=\kappa^s(Q_n;Q_1)=n-1$ \cite{02}. By the definition, we have 
\[
D_f(Q_n;Q_1)= \text{max} \{D(Q_n-\mathcal{F}_1)\,|\,\mathcal{F}_1 \text{ is a set of } Q_1\text{-structures and } |\mathcal{F}_1|\le n-2\},
\]
where every $Q_1$-structure is isomorphic to $K_2$. And 
\[
D^s_f(Q_n;Q_1)=\text{max}\{D(Q_n-\mathcal{F}_2)\,|\,\mathcal{F}_2 \text{ is a set of } Q_1\text{-substructures and }  |\mathcal{F}_2|\le n-2\},
\] 
where each element of $\mathcal{F}_2$ is independently isomorphic to either $K_2$ or $K_1$.

We provide some lemmas for later use.

\begin{mylemma}\label{lemma3.1} Let $m\le n-3$ and $| \mathcal{F}^n_m|\le n-1$. For any two symmetric vertices $\textbf{u}$ and $\textbf{v}$ in ${Q_n}-\mathcal{F}^n_m$, there exists a pair of vertices $(\textbf{u})^{j}$ and $(\textbf{v})^{j}$ in ${Q_n}-\mathcal{F}^n_m$ for some $j\in \{{1,2,\ldots,n}\}$.
\end{mylemma}
\begin{proof} Let $(\textbf{u})^{j}$ and $(\textbf{v})^{k}$ respectively be neighbors of $\textbf{u}$ and $\textbf{v}$ in $Q_n$, where $j,k\in \{{1,2,\ldots,n}\}$. Then $H_{Q_n}((\textbf{u})^{j}$, $(\textbf{v})^{k})=n$ if $j=k$, and $H_{Q_n}((\textbf{u})^{j}$, $(\textbf{v})^{k})=n-2$ if $j\ne k$. Combining this with the condition $m\le n-3$, we infer that no subcube in $\mathcal{F}^n_m$ can contain both $(\textbf{u})^{j}$ and $(\textbf{v})^{k}$ simultaneously. By Corollary~\ref{corollary2.6}, no subcube in $\mathcal{F}^n_m$ can contain both $(\textbf{u})^{j}$ and $(\textbf{u})^{h}$ for $j\ne h$ simultaneously. The same holds for $(\textbf{v})^{j}$ and $(\textbf{v})^{h}$ for $j\ne h$. This implies that the removal of any subcube in $\mathcal{F}^n_m$ reduces the neighbors of $\textbf{u}$ or $\textbf{v}$ by at most one. Note that $d_{Q_n}(\textbf{u})=d_{Q_n}(\textbf{v})=n$. However, $| \mathcal{F}^n_m|\le n-1$. So there must exist a pair of vertices $(\textbf{u})^{j}$ and $(\textbf{v})^{j}$ in ${Q_n}-\mathcal{F}^n_m$.
\end{proof}

\begin{mytheorem}\label{theorem3.3}
	$D^s_f(Q_3;Q_1)=3$.
\end{mytheorem}
\begin{proof} By Lemma~\ref{lemma2.3}, $\kappa^s(Q_3;Q_1) = 2$. Thus, we need to consider the event $| \mathcal{F}^3_1|\le \kappa^s(Q_3;Q_1)-1=1$. By Lemma~\ref{lemma3.2},
$D(Q_3-\mathcal{F}^3_1)=3$ if $| F^3_1|=0$. Bellow, we suppose that $| F^3_0|=0$ and $| F^3_1|=1$. Since $Q_3$ is vertex transitive and edge transitive, we may assume that $F^3_1=\{\{000,001\}\}$ is a faulty $Q_1$-structure in $Q_3$. From \Cref{fig:4}, we get that the diameter of $Q_3-F^3_1$ is $3$, and so $D^s_f(Q_3;Q_1)=3$.
\end{proof}

\begin{mylemma}\label{lemma3.4}
	For $n\ge 4$, $D_f(Q_n;Q_1)\ge n+1$ and $D^s_f(Q_n;Q_1)\ge n+1$.
\end{mylemma}
\begin{proof} Let $\textbf{x}=00 \cdots 0$ and $\textbf{z}=(\textbf{x})^{1}$. Then, we set $\mathcal{F}^n_1=\{\{(\textbf{x})^{i}, (\textbf{z})^{i}\}\mid2\le i\le n-1\}$. Obviously, $| \mathcal{F}^n_1|=n-2$, ${Q^0_n}-\mathcal{F}^n_1$ is disconnected and one of components of ${Q^0_n}-\mathcal{F}^n_1$ contains $\{\textbf{x}, \textbf{z}\}$. Further, we let $\textbf{y}=11 \cdots 10$. Clearly, $\{\textbf{x}, \textbf{z}\}$ and $\textbf{y}$ are in the distinct components of ${Q^0_n}-\mathcal{F}^n_1$. Since $Q_n-\mathcal{F}^n_1$ is connected and $N_{Q_n-\mathcal{F}^n_1}(\textbf{x})={\{\textbf{z},(\textbf{x})^{n}\}}$, there are two paths $\langle \textbf{x}, (\textbf{x})^{n}, \textit{P}_1, \textbf{y} \rangle$ and $\langle \textbf{x}, \textbf{z}, (\textbf{z})^{n}, \textit{P}_2, \textbf{y} \rangle$ between $\textbf{x}$ and $\textbf{y}$ in $Q_n-\mathcal{F}^n_1$. Note that $H_{Q_n}((\textbf{x})^{n},\textbf{y})=n$ and $H_{Q_n}((\textbf{z})^{n},\textbf{y})=n-1$, so we have $l(P_1)\ge n$ and $l(P_2)\ge n-1$, respectively.
Thus, the length of any path between $\textbf{x}$ and $\textbf{y}$ is at least $n+1$. This implies $D_f(Q_n;Q_1)\ge n+1$. Moreover, since $D^s_f(Q_n;Q_1) \ge D_f(Q_n;Q_1)$, $D^s_f(Q_n;Q_1)\ge n+1$.
\end{proof}

\begin{figure}
		\centering
	\includegraphics[height=5cm]{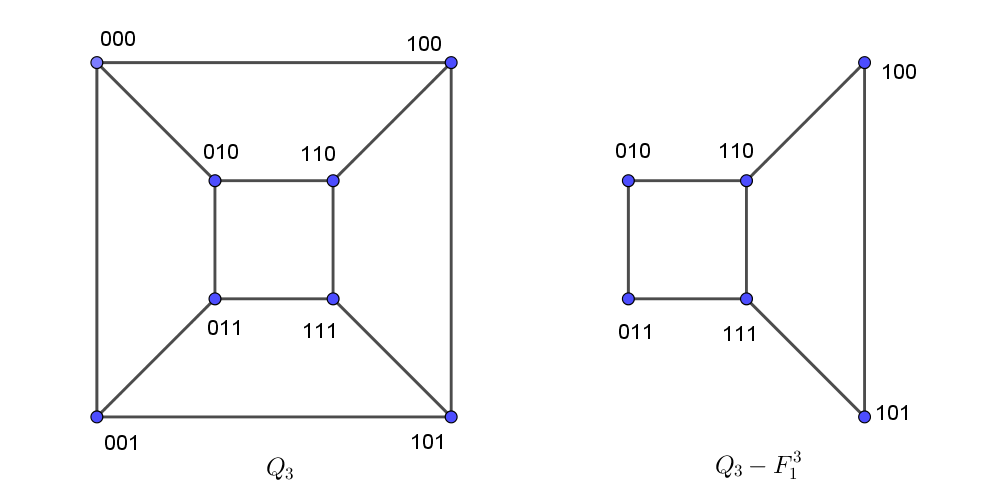}
\caption[Fig.3]{An illustration of the proof in Theorem~\ref{theorem3.3}}.
\label{fig:4}
\end{figure}

\begin{mylemma}\label{lemma3.5}
	$D^s_f(Q_4;Q_1)\le 5$.
\end{mylemma}
\begin{proof} By Lemma~\ref{lemma2.3}, $\kappa^s(Q_4;Q_1) = 3$. Thus, we need to consider the event $| \mathcal{F}^4_1|\le 2$. By Lemma~\ref{lemma2.2}, $D_f(Q_4)=5$. If $|F^4_1|<2$, then $\mathcal{F}^4_1$ contains at most $3$ vertices and we have
	$D(Q_4-\mathcal{F}^4_1)\le D_f(Q_4)=5$. Thus, we suppose that  $| F^4_0|=0$ and $| F^4_1|=2$. Since $Q_4$ is vertex transitive and edge transitive, we assume that $\textbf{x}\neq (\textbf{y})^4$ for each $\{\textbf{x}, \textbf{y}\}\in F^4_1$.
	Thus, $| B^4_1|=0$ and $F^4_1=A^4_{1,0}\cup A^4_{1,1}$. Without loss of generality, we assume
	$| A^4_{1,0}|\ge | A^4_{1,1}|$. Let $\textbf{u}$ and $\textbf{v}$ be any two vertices in $Q_4-F^4_1$, then we have the following cases:

\noindent \textbf{Case 1}. $| A^4_{1,0}|=2$. Since $D(Q_3)=3$, $d_{Q_4-F^4_1}(\textbf{u},\textbf{v})\le 3$ if $\textbf{u}$,$\textbf{v}\in {Q^1_4}$. If $\textbf{u}$,$\textbf{v}\in {Q^0_4}$, then there exists a path $\langle \textbf{u}, (\textbf{u})^{4}, \textit{P}_s, (\textbf{v})^{4}, \textbf{v} \rangle$ between $\textbf{u}$ and $\textbf{v}$ in $Q_4-F^4_1$, where $P_s$ is in ${Q^1_4}$. Since $l(\textit{P}_s)\le D({Q^1_4})=3$, $d_{Q_4-F^4_1}(\textbf{u},\textbf{v})\le 5$. Obviously, if  $\textbf{u}\in {Q^0_4}$ and $\textbf{v}\in {Q^1_4}$, then there exists a path $\langle \textbf{u}, (\textbf{u})^{4}, \textit{P}_s, \textbf{v} \rangle$ between $\textbf{u}$ and $\textbf{v}$ in $Q_4-F^4_1$, where $P_s$ is in ${Q^1_4}$. Then $d_{Q_4-F^4_1}(\textbf{u},\textbf{v})\le 4$.
	
\noindent 	\textbf{Case 2}. $| A^4_{1,0}|=1$ and $| A^4_{1,1}|=1$. By Theorem~\ref{theorem3.3}, $d_{Q_4-F^4_1}(\textbf{u},\textbf{v})\le 3$ if $\textbf{u}$,$\textbf{v}\in {Q^0_4}$ or $\textbf{u}$,$\textbf{v}\in {Q^1_4}$. So we only need to consider that $\textbf{u}\in {Q^0_4}$ and $\textbf{v}\in {Q^1_4}$. Since $N_{{Q^0_4}-A^4_{1,0}}(\textbf{u})\ge 2$, we set $\{\textbf{u}_1,\textbf{u}_2\} \subseteq N_{{Q^0_4}-A^4_{1,0}}(\textbf{u})$. Since $((\textbf{u}_1)^4)^i\ne (\textbf{u}_2)^4$ for each $i\in\{1,2,3\}$, there must exist $(\textbf{u}_1)^4\in {Q^1_4}-A^4_{1,1}$ or $(\textbf{u}_2)^4\in {Q^1_4}-A^4_{1,1}$. We suppose the former, i.e.,  $(\textbf{u}_1)^4\in {Q^1_4}-A^4_{1,1}$. Thus, there exists a path $\langle \textbf{u}, \textbf{u}_1, (\textbf{u}_1)^4,\textit{P}_s, \textbf{v} \rangle$ between $\textbf{u}$ and $\textbf{v}$ in $Q_4-F^4_1$, where $P_s$ is in ${Q^1_4}-A^4_{1,1}$. Clearly, $l(\textit{P}_s)\le D({Q^1_4-A^4_{1,1}})=3$, then $d_{Q_4-F^4_1}(\textbf{u},\textbf{v})\le 5$.
\end{proof}

\begin{mylemma}\label{lemma3.6}
	$D^s_f(Q_n;Q_1)\le n+1$ for $n\ge 4$.
\end{mylemma}

\begin{proof} We prove this lemma by induction on $n$. By Lemma~\ref{lemma3.5}, the lemma holds for $n=4$. Thus, we assume that this lemma holds for $4\le k\le n-1$, i.e., $D^s_f(Q_k;Q_1)\le k+1$ for $4 \le k \le n-1$.  Note that $\kappa^s(Q_n;Q_1)$=$n-1$ by Lemma~\ref{lemma2.3}, then $Q_n-\mathcal{F}^n_1$ is connected for $| \mathcal{F}^n_1|\le n-2$. And let $\textbf{u}$ and $\textbf{v}$ be any two vertices in $Q_n-\mathcal{F}^n_1$.
	
\noindent 	\textbf{Case 1}. $\textbf{u}$ and $\textbf{v}$ are symmetric. We may assume $\textbf{u}=00 \cdots 0$ and $\textbf{v}=11 \cdots 1$. Since $n\ge 5$ and $| \mathcal{F}^n_1|\le n-2$, by  Lemma~\ref{lemma3.1}, there must exist a pair of vertices $(\textbf{u})^{j}$ and $(\textbf{v})^{j}$ in $Q_n-\mathcal{F}^n_1$  for some $j\in \{1,2, \ldots,n\}$. Thus we may assume $(\textbf{u})^{n},(\textbf{v})^{n}\in Q_n-\mathcal{F}^n_1$.
	Note that $\mathcal{F}^n_1$ contains at most $2(n-2)$ vertices. So we may assume ${Q^0_n}$ has at most $n-2$ faulty  vertices. Then there exists a path $\langle \textbf{u}, \textit{P}_s, (\textbf{v})^n, \textbf{v} \rangle$ between $\textbf{u}$ and $\textbf{v}$ in ${Q_n}-\mathcal{F}^n_1$, where $\textit{P}_s$ is in ${Q^0_n}-\mathcal{F}^n_1$. By Lemma~\ref{lemma2.2}, we have $ D^s_f({Q^0_n};Q_1)\le n$. Thus, $l(P_s)\le n$ and $d_{{Q_n}-\mathcal{F}^n_1}(\textbf{u}, \textbf{v})\le n+1$.
	
\noindent 	\textbf{Case 2}.  $\textbf{u}$ and $\textbf{v}$ are unsymmetric.  Without loss of generality, we may assume $\textbf{u}$, $\textbf{v}$ $\in$ ${Q^0_n}$.
	
\noindent  \textbf{Case 2.1}. $| \mathcal{A}^n_{1,1}|\ge 1$. Then $| \mathcal{A}^n_{1,0}|+|\mathcal{B}^n_1| \le n-3$. Thus, there exists a path $\langle \textbf{u},\textit{P}_s, \textbf{v} \rangle$ between $\textbf{u}$ and $\textbf{v}$ in $Q_n-\mathcal{F}^n_1$, where $P_s$ is in ${Q^0_n}-\mathcal{A}^n_{1,0}-\mathcal{B}^n_1$. By the induction hypothesis, we infer $l(\textit{P}_s)\le n$, and so we get $d_{Q_n-\mathcal{F}^n_1}(\textbf{u}$, $\textbf{v})\le n$.
	
\noindent 	\textbf{Case 2.2}. $| \mathcal{A}^n_{1,1}|=0$. If $| \mathcal{A}^n_{1,0}|=0$, then ${Q^0_n}-\mathcal{A}^n_{1,0}-\mathcal{B}^n_1={Q^0_n}-\mathcal{B}^n_1$. Thus, there exists a path $\langle \textbf{u},\textit{P}_s,\textbf{v}\rangle$ between $\textbf{u}$ and $\textbf{v}$ in $Q_n-\mathcal{F}^n_1$, where $P_s$ is in ${Q^0_n}-\mathcal{B}^n_1$. Since $l(\textit{P}_s)\le D({Q^0_n}-\mathcal{B}^n_1)\le D_f(Q_{n-1})=n$, $d_{Q_n-\mathcal{F}^n_1}(\textbf{u}$, $\textbf{v})\le n$. If $| \mathcal{A}^n_{1,0}|\ge 1$, then there exists a path $\langle \textbf{u}, (\textbf{u})^{n},\textit{P}_s, (\textbf{v})^{n},\textbf{v}\rangle$ between $\textbf{u}$ and $\textbf{v}$ in $Q_n-\mathcal{F}^n_1$, where $P_s$ is in ${Q^1_n}-\mathcal{B}^n_1$ and $|\mathcal{B}^n_1| \le n-3$. By Lemma~\ref{lemma3.2}, we have $D({Q^1_n}-\mathcal{B}^n_1)=n-1$. Then $l(\textit{P}_s)\le n-1$, and so $d_{Q_n-\mathcal{F}^n_1}(\textbf{u}$, $\textbf{v})\le n+1$.
\end{proof}

Combining  Theorem~\ref{theorem3.3}, Lemma~\ref{lemma3.4}, and Lemma~\ref{lemma3.6} with the fact $D^s_f(G;W)\ge D_f(G;W)$, we have the following result.

\begin{mytheorem}\label{theorem3.7}
	$D^s_f(Q_3;Q_1)=D_f(Q_3;Q_1)=3$ and $D^s_f(Q_n;Q_1)=D_f(Q_n;Q_1)=n+1$ if $n \geq 4$.
\end{mytheorem}

\section{$Q_m$-structure fault diameter}

Sabir et al.~proved that  $\kappa(Q_n;Q_m)=\kappa^{sc}(Q_n;Q_m)=n-m$ \cite{04}. By the definition, we have 
\[
D_f(Q_n;Q_m)=\text{max}\{D(Q_n-\mathcal{F}_1)\,|\,\mathcal{F}_1 \text{ is a set of } Q_m\text{-structures and } |\mathcal{F}_1|\le n-m-1\},
\]
where every $Q_m$-structure is isomorphic to $Q_m$. And 
\[
D^{sc}_f(Q_n;Q_m)=\text{max}\{D(Q_n-\mathcal{F}_2)\,|\,\mathcal{F}_2\text{ is a set of } Q_m\text{-subcubes and } |\mathcal{F}_2|\le n-m-1\}, 
\]
where each element of $\mathcal{F}_2$ is independently isomorphic to some $Q_i$, $0\le i\le m$. Especially, if every $Q_m$-subcube in $\mathcal{F}_2$ is isomorphic to $Q_m$, then $D^{sc}_f(Q_n;Q_m)=D_f(Q_n;Q_m)$. This infer that $D^{sc}_f(Q_n;Q_m)\ge D_f(Q_n;Q_m)$.	Let's first determine $Q_m$-subcube fault diameter $D_f^{sc}(Q_n;Q_m)$.

\begin{mytheorem}\label{theorem3.20}
	$D^{sc}_f(Q_{m+2};Q_m)=m+2$ for $m\ge 0$.
\end{mytheorem}

\begin{proof} Firstly, we prove $D^{sc}_f(Q_{m+2};Q_m)\le m+2$ by induction on $m$. Clearly, $D^{sc}_f(Q_2;Q_0)=D_f(Q_2)=2$ and $D^{sc}_f(Q_3;Q_1)=D^{s}_f(Q_3;Q_1)=3$, so the lemma is true for $m=0,1$. In the following, we assume that the statement holds for $1\le k\le m-1$, i.e., $D^{sc}_f(Q_{k+2};Q_k)\le k+2$ for $1\le k\le m-1$. Note that $\kappa^{sc}(Q_{m+2};Q_m) = 2$ by Lemma~\ref{lemma2.4}. For $| \mathcal{F}^{m+2}_m|\le 1$, without loss of generality, we may assume $\mathcal{F}^{m+2}_m\subseteq \bigcup^m_{i=0}S_i({Q^0_{m+2}})$. Let $\textbf{u}$ and $\textbf{v}$ be any two vertices in $Q_{m+2}-\mathcal{F}^{m+2}_m$.

\noindent  \textbf{Case 1}. $| F^{m+2}_m|=0$. Since $\mathcal{F}^{m+2}_m=F^{m+2}_m\cup\mathcal{F}^{m+2}_{m-1}$, $| \mathcal{F}^{m+2}_{m-1}|\le 1$. By the induction hypothesis, we have $D(Q^0_{m+2}-\mathcal{F}^{m+2}_{m-1})\le D^{sc}_f(Q_{m+1};Q_{m-1})\le m+1$.
Then $d_{Q_{m+2}-\mathcal{F}^{m+2}_m}(\textbf{u}$, $\textbf{v})\le m+1$ if $\textbf{u},\textbf{v}\in Q^0_{m+2}$. Obviously,  $d_{Q_{m+2}-\mathcal{F}^{m+2}_m}(\textbf{u}$, $\textbf{v})\le m+1$ if $\textbf{u},\textbf{v}\in Q^1_{m+2}$. Since $(\textbf{u},(\textbf{u})^{m+2})\in E(Q_{m+2}-\mathcal{F}^{m+2}_m)$, there exists a path $\langle \textbf{u}, (\textbf{u})^{m+2}, \textit{P}_s,\textbf{v}\rangle$ between $\textbf{u}\in Q^0_{m+2}$ and  $\textbf{v}\in Q^1_{m+2}$ in $Q_{m+2}-\mathcal{F}^{m+2}_m$, where $P_s$ is in ${Q^1_{m+2}}$. Note that $l(\textit{P}_s)\le D(Q^1_{m+2})=m+1$, then $d_{Q_{m+2}-\mathcal{F}^{m+2}_m}(\textbf{u}$, $\textbf{v})\le m+2$.

\noindent \textbf{Case 2}. $| F^{m+2}_m|=1$. Since $Q^0_{m+2}-F^{m+2}_m$ is connected and isomorphic to $Q_{m}$, $d_{Q_{m+2}-\mathcal{F}^{m+2}_m}(\textbf{u}$, $\textbf{v})\le m$ if $\textbf{u},\textbf{v}\in Q^0_{m+2}$. With similar argument used in Case 1, one can get that $d_{Q_{m+2}-\mathcal{F}^{m+2}_m}(\textbf{u}$, $\textbf{v})\le m+1$ if $\textbf{u},\textbf{v}\in Q^1_{m+2}$, and $d_{Q_{m+2}-\mathcal{F}^{m+2}_m}(\textbf{u}$, $\textbf{v})\le m+2$ if $\textbf{u}\in Q^0_{m+2}$ and  $\textbf{v}\in Q^1_{m+2}$.

Since there are $2^{m+1}$ pairs of symmetric vertices in $Q_{m+2}$ and $2^{m+1}> 2^m$, we have $D^{sc}_f(Q_{m+2};\\Q_m)\ge m+2$ by Lemma~\ref{lemma2.7}.

In summary, we get $D^{sc}_f(Q_{m+2};Q_m)=m+2$.
\end{proof}

\begin{mylemma}\label{lemma3.21}
	$D^{sc}_f(Q_{m+3};Q_m)\le m+4$ for $m\ge 0$.
\end{mylemma}

\begin{proof}
	We prove this lemma by induction on $m$. Clearly,  $D^{sc}_f(Q_3;Q_0)=D_f(Q_3)=4$ and $D^{sc}_f(Q_4;Q_1)=D^s_f(Q_4;Q_1)\le 5$, so the result holds for $m=0,1$. In the following, we assume that the statement holds for $1\le k\le m-1$, i.e., $D^{sc}_f(Q_{k+3};Q_k)\le k+4$ for $1\le k\le m-1$. Note that $\kappa^{sc}(Q_{m+3};Q_m) = 3$. So we assume $| \mathcal{F}^{m+3}_m|\le 2$ and let $\textbf{u}$ and $\textbf{v}$ be any two vertices in $Q_{m+3}-\mathcal{F}^{m+3}_m$.
	
\noindent 	\textbf{Case 1}.  $\textbf{u}$ and $\textbf{v}$ are symmetric. Without loss of generality, we take $\textbf{u}=00 \cdots 0$ and $\textbf{v}=11 \cdots 1$. For $m\ge 2$ and $| \mathcal{F}^{m+3}_m|\le 2\le m+2$, by Lemma~\ref{lemma3.1},
	there exists a pair of vertices $(\textbf{u})^{j}$ and $(\textbf{v})^{j}$ in $Q_{m+3}-\mathcal{F}^{m+3}_m$. For convenience, we may assume $(\textbf{u})^{m+3},(\textbf{v})^{m+3}\in Q_{m+3}-\mathcal{F}^{m+3}_m$.
	
\noindent 	\textbf{Case 1.1}. $| \mathcal{A}^{m+3}_{m,0}|+| \mathcal{B}^{m+3}_m|=2$. Then $| \mathcal{A}^{m+3}_{m,1}|=0$ and
	$| \mathcal{B}^{m+3}_m|\le 2$. Thus,
	there exists a path $\langle \textbf{u}, (\textbf{u})^{m+3},\textit{P}_s, \textbf{v} \rangle$ between $\textbf{u}$ and $\textbf{v}$ in $Q_{m+3}-\mathcal{F}^{m+3}_m$, where $\textit{P}_s$ is in ${Q^1_{m+3}}-\mathcal{B}^{m+3}_m$. By the induction hypothesis, we have
		$l(\textit{P}_s)
		 \le D({Q^1_{m+3}}-\mathcal{B}^{m+3}_m) \le D^{sc}_f(Q_{m+2};Q_{m-1})
		 \le m+3$.
	Thus $d_{Q_{m+3}-\mathcal{F}^{m+3}_m}(\textbf{u}$, $\textbf{v})\le m+4$.
	
\noindent 	 \textbf{Case 1.2}. $| \mathcal{A}^{m+3}_{m,0}|+| \mathcal{B}^{m+3}_m|\le 1$. Then
	there exists a path $\langle \textbf{u},\textit{P}_s,(\textbf{v})^{m+3}, \textbf{v} \rangle$ between $\textbf{u}$ and $\textbf{v}$ in $Q_{m+3}-\mathcal{F}^{m+3}_m$, where $\textit{P}_s$ is in ${Q^0_{m+3}}-\mathcal{A}^{m+3}_{m,0}-\mathcal{B}^{m+3}_m$. By Theorem~\ref{theorem3.20}, $l(\textit{P}_s)
		 \le D({Q^0_{m+3}}-\mathcal{A}^{m+3}_{m,0}-\mathcal{B}^{m+3}_m) \le D^{sc}_f(Q_{m+2};Q_m) = m+2$.
	So we have  $d_{Q_{m+3}-\mathcal{F}^{m+3}_m}(\textbf{u}$, $\textbf{v})\le m+3$.
	
\noindent 	\textbf{Case 2}.  $\textbf{u}$ and $\textbf{v}$ are unsymmetric. We may assume $\textbf{u}$, $\textbf{v}$ $\in$ ${Q^0_n}$.

\noindent  \textbf{Case 2.1.}. $| \mathcal{A}^{m+3}_{m,0}|+| \mathcal{B}^{m+3}_m|=2$ and $| \mathcal{B}^{m+3}_m|=2$. Then $| \mathcal{A}^{m+3}_{m,0}|=0$. Thus, there exists a path $\langle \textbf{u},\textit{P}_s,\textbf{v} \rangle$ between $\textbf{u}$ and $\textbf{v}$ in $Q_{m+3}-\mathcal{F}^{m+3}_m$, where $\textit{P}_s$ is in ${Q^0_{m+3}}-\mathcal{B}^{m+3}_m$. By the induction hypothesis, $l(\textit{P}_s)
		\le D({Q^0_{m+3}}-\mathcal{B}^{m+3}_m) \le D^{sc}_f(Q_{m+2};Q_{m-1}) \le m+3$.	So $d_{Q_{m+3}-\mathcal{F}^{m+3}_m}(\textbf{u}$, $\textbf{v})\le m+3$.
	
\noindent  \textbf{Case 2.2.}. $| \mathcal{A}^{m+3}_{m,0}|+| \mathcal{B}^{m+3}_m|=2$ and $| \mathcal{B}^{m+3}_m|\le 1$. Note that  $| \mathcal{A}^{m+3}_{m,1}|=0$. Thus,
	there exists a path $\langle \textbf{u}, (\textbf{u})^{m+3},\textit{P}_s, (\textbf{v})^{m+3},\textbf{v} \rangle$ between $\textbf{u}$ and $\textbf{v}$ in $Q_{m+3}-\mathcal{F}^{m+3}_m$, where $\textit{P}_s$ is in ${Q^1_{m+3}}-\mathcal{B}^{m+3}_m$. By Theorem~\ref{theorem3.20}, we have $l(\textit{P}_s)
		\le D({Q^1_{m+3}}-\mathcal{B}^{m+3}_m)
		 \le D^{sc}_f(Q_{m+2};Q_m)
		=m+2$.
	So $d_{Q_{m+3}-\mathcal{F}^{m+3}_m}(\textbf{u}$, $\textbf{v})\le m+4$.
	
\noindent  \textbf{Case 2.3.} $| \mathcal{A}^{m+3}_{m,0}|+| \mathcal{B}^{m+3}_m|\le 1$. There exists a path $\langle \textbf{u},\textit{P}_s, \textbf{v} \rangle$ between $\textbf{u}$ and $\textbf{v}$ in $Q_{m+3}-\mathcal{F}^{m+3}_m$, where $\textit{P}_s$ is in ${Q^0_{m+3}}-\mathcal{A}^{m+3}_{m,0}-\mathcal{B}^{m+3}_m$. By Case 1.2,  we get $l(\textit{P}_s)\le m+2$, and thus $d_{Q_{m+3}-\mathcal{F}^{m+3}_m}(\textbf{u}$, $\textbf{v})\le m+2$.
\end{proof}

\begin{mylemma}\label{lemma3.22} Let $m\ge 0$ and $n\ge m+3$. If $D^{sc}_f(Q_{n-1};Q_m)\leqslant n$, then $D(Q_{n-1}-\mathcal{F}^{n-1}_m)\le n-1$ for $| \mathcal{F}^{n-1}_m|\le n-m-3$.
\end{mylemma}

\begin{proof} For $| \mathcal{F}^{n-1}_m|\le n-m-3$, we prove $D(Q_{n-1}-\mathcal{F}^{n-1}_m)\le n-1$ by induction on $n$. Obviously, if $| \mathcal{F}^{m+2}_m|= 0$, we have $D(Q_{m+2}-\mathcal{F}^{m+2}_m)=D(Q_{m+2})=m+2$. Next, we prove $D(Q_{m+3}-\mathcal{F}^{m+3}_m)\le m+3$ for $| \mathcal{F}^{m+3}_m|\le 1$. Without loss of generality, we may assume $\mathcal{F}^{m+3}_m\subseteq \bigcup^m_{i=0}S_i({Q^0_{m+3}})$. Let $\textbf{x}$ and $\textbf{y}$ be two vertices in $Q_{m+3}-\mathcal{F}^{m+3}_m$. By Theorem~\ref{theorem3.20}, $d_{Q_{m+3}-\mathcal{F}^{m+3}_m}(\textbf{x}$, $\textbf{y})\le m+3$ if $\textbf{x},\textbf{y}\in Q^0_{m+3}$ or $\textbf{x},\textbf{y}\in Q^1_{m+3}$. Since $(\textbf{x},(\textbf{x})^{m+3})\in E(Q_{m+3}-\mathcal{F}^{m+3}_m)$, $d_{Q_{m+3}-\mathcal{F}^{m+3}_m}(\textbf{x}$, $\textbf{y})\le m+3$ for $\textbf{x}\in Q^0_{m+3}$ and $\textbf{y}\in Q^1_{m+3}$. This means that the result holds for $n=m+3$ and $n=m+4$.
Thus, we assume that the result also holds for $m+4\le k\le n-1$, i.e., $D(Q_{k-1}-\mathcal{F}^{k-1}_m)\le k-1$ for $m+4\le k\le n-1$. Let $\textbf{u}$ and $\textbf{v}$ be any two vertices in $Q_{n-1}-\mathcal{F}^{n-1}_m$.

\noindent \textbf{Case 1}.  $\textbf{u}$ and $\textbf{v}$ are symmetric. We set $\textbf{u}=00 \cdots 0$ and $\textbf{v}=11 \cdots 1$. For $m\le n-5$ and $| \mathcal{F}^{n-1}_m|\le n-m-3$, by Lemma~\ref{lemma3.1}, there exists a pair of vertices $(\textbf{u})^{j}$ and $(\textbf{v})^{j}$ in $Q_{n-1}-\mathcal{F}^{n-1}_m$. We may assume $(\textbf{u})^{n-1},(\textbf{v})^{n-1}\in Q_{n-1}-\mathcal{F}^{n-1}_m$.

\noindent \textbf{Case 1.1}. $| \mathcal{A}^{n-1}_{m,0}|+| \mathcal{B}^{n-1}_m|=n-m-3$. Then $| \mathcal{A}^{n-1}_{m,1}|=0$ and
$| \mathcal{B}^{n-1}_m|\le n-m-3$. This means that
there exists a path $\langle \textbf{u}, (\textbf{u})^{n-1},\textit{P}_s, \textbf{v} \rangle$ between $\textbf{u}$ and $\textbf{v}$ in $Q_{n-1}-\mathcal{F}^{n-1}_m$, where $\textit{P}_s$ is in ${Q^1_{n-1}}-\mathcal{B}^{n-1}_m$. By the induction hypothesis, we have $
	l(\textit{P}_s)
	 \le D({Q^1_{n-1}}-\mathcal{B}^{n-1}_m) \le D(Q_{n-2}-\mathcal{F}^{n-2}_{m-1})  \le n-2$.
Thus $d_{Q_{n-1}-\mathcal{F}^{n-1}_m}(\textbf{u}$, $\textbf{v})\le n-1$.

\noindent\textbf{Case 1.2}. $| \mathcal{A}^{n-1}_{m,0}|+| \mathcal{B}^{n-1}_m|\le n-m-4$. Then there exists a path $\langle \textbf{u},\textit{P}_s,(\textbf{v})^{n-1}, \textbf{v} \rangle$ between $\textbf{u}$ and $\textbf{v}$ in $Q_{n-1}-\mathcal{F}^{n-1}_m$, where $\textit{P}_s$ is in ${Q^0_{n-1}}- \mathcal{A}^{n-1}_{m,0}- \mathcal{B}^{n-1}_m$. By the induction hypothesis,
we have $l(\textit{P}_s)  \le D({Q^0_{n-1}}-\mathcal{A}^{n-1}_{m,0}- \mathcal{B}^{n-1}_m) \le  D(Q_{n-2}-\mathcal{F}^{n-2}_{m}) \le n-2$.
So $d_{Q_{n-1}-\mathcal{F}^{n-1}_m}(\textbf{u}$, $\textbf{v})\le n-1$.

\noindent \textbf{Case 2}.  $\textbf{u}$ and $\textbf{v}$ are unsymmetric. We may assume $\textbf{u}$, $\textbf{v}$ $\in$ ${Q^0_{n-1}}$.
Note that $| \mathcal{A}^{n-1}_{m,0}|+| \mathcal{B}^{n-1}_m|\le n-m-3$ and $D^{sc}_f(Q_{n-1};Q_m)\le n$. Then there exists a path $\langle \textbf{u},\textit{P}_s, \textbf{v} \rangle$ between $\textbf{u}$ and $\textbf{v}$ in $Q_{n-1}-\mathcal{F}^{n-1}_m$, where $\textit{P}_s$ is in ${Q^0_{n-1}}-\mathcal{A}^{n-1}_{m,0}-\mathcal{B}^{n-1}_m$. We have
$d_{Q_{n-1}-\mathcal{F}^{n-1}_m}(\textbf{u}, \textbf{v})
	 \le D(Q^0_{n-1}-\mathcal{A}^{n-1}_{m,0}-\mathcal{B}^{n-1}_m) \le D^{sc}_f(Q_{n-2};Q_m) \le n-1$.
\end{proof}

\begin{mylemma}\label{lemma3.23}
	$D^{sc}_f(Q_n;Q_m)\le n+1$ for $m\ge 0$ and $n\ge m+3$.
\end{mylemma}

\begin{proof}  We prove this lemma by induction on $n$. By Lemma~\ref{lemma3.21}, this lemma holds for $n=m+3$. Thus, we assume that this lemma holds for $m+3\le k\le n-1$, i.e., $D^{sc}_f(Q_k;Q_m)\le k+1$ for $m+3\le k\le n-1$. Note that $\kappa^{sc}(Q_n;Q_m) = n-m$. Then, for $| \mathcal{F}^n_m|\le n-m-1$, let $\textbf{u}$ and $\textbf{v}$ be any two vertices in $Q_n-\mathcal{F}^n_m$.

\noindent  \textbf{Case 1}. $\textbf{u}$ and $\textbf{v}$ are symmetric. Note that $m\le n-4$ and $| \mathcal{F}^n_m|\le n-m-1$. By Lemma~\ref{lemma3.1}, there must exist a pair of vertices $(\textbf{u})^{j}$ and $(\textbf{v})^{j}$ in $Q_n-\mathcal{F}^n_m$ for some $j\in \{1,2, \ldots,n\}$, thus we may assume $(\textbf{u})^{n},(\textbf{v})^{n}\in Q_n-\mathcal{F}^n_m$.

\noindent \textbf{Case 1.1}. $| \mathcal{A}^n_{m,0}|+| \mathcal{B}^n_m|=n-m-1$. Then $| \mathcal{A}^n_{m,1}|=0$ and $| \mathcal{B}^n_m|\le n-m-1$. This implies that
there exists a path $\langle \textbf{u}, (\textbf{u})^{n},\textit{P}_s, \textbf{v} \rangle$ between $\textbf{u}$ and $\textbf{v}$ in $Q_n-\mathcal{F}^n_m$, where $\textit{P}_s$ is in ${Q^1_n}-\mathcal{B}^n_m$. By the induction hypothesis, we have $
	l(\textit{P}_s)
	 \le D({Q^1_n}-\mathcal{B}^n_m) \le D^{sc}_f(Q_{n-1};Q_{m-1}) \le n$.
So $d_{Q_n-\mathcal{F}^n_m}(\textbf{u}$, $\textbf{v})\le n+1$.

\noindent \textbf{Case 1.2}. $| \mathcal{A}^n_{m,0}|+| \mathcal{B}^n_m|\le n-m-2$. Then
there exists a path $\langle \textbf{u},\textit{P}_s,(\textbf{v})^{n}, \textbf{v} \rangle$ between $\textbf{u}$ and $\textbf{v}$ in $Q_n-\mathcal{F}^n_m$, where $\textit{P}_s$ is in ${Q^0_n}-\mathcal{A}^n_{m,0}-\mathcal{B}^n_m$. By the induction hypothesis,
we have $l(\textit{P}_s)
	 \le D({Q^0_n}-\mathcal{A}^n_{m,0}-\mathcal{B}^n_m) \le D^{sc}_f(Q_{n-1};Q_m) \le n$.
Thus, $d_{Q_n-\mathcal{F}^n_m}(\textbf{u}$, $\textbf{v})\le n+1$.

\noindent \textbf{Case 2}.  $\textbf{u}$ and $\textbf{v}$ are unsymmetric.  Without loss of generality, we may assume $\textbf{u}$, $\textbf{v}$ $\in$ ${Q^0_n}$.

\noindent \textbf{Case 2.1.} $| \mathcal{A}^n_{m,0}|+| \mathcal{B}^n_m|=n-m-1$, and $| \mathcal{A}^n_{m,0}|\ge 1$. Then $| \mathcal{A}^n_{m,1}|=0$ and $| \mathcal{B}^n_m|\le n-m-2$. Thus, there exists a path $\langle \textbf{u}, (\textbf{u})^{n},\textit{P}_s,(\textbf{v})^{n}, \textbf{v} \rangle$ between $\textbf{u}$ and $\textbf{v}$ in $Q_n-\mathcal{F}^n_m$, where $\textit{P}_s$ is in ${Q^1_n}-\mathcal{B}^n_m$. By the induction hypothesis, we have $D^{sc}_f(Q_{n-1};Q_{m-1})\le n$.
Then, by Lemma~\ref{lemma3.22}, $
	l(\textit{P}_s)
	 \le D({Q^1_n}-\mathcal{B}^n_m) \le D(Q_{n-1}-\mathcal{F}^{n-1}_{m-1}) \le n-1$.
And so $d_{Q_n-\mathcal{F}^n_m}(\textbf{u}$, $\textbf{v})\le n+1$.

\noindent \textbf{Case 2.2.} $| \mathcal{A}^n_{m,0}|+| \mathcal{B}^n_m|=n-m-1$, and $| \mathcal{A}^n_{m,0}|=0$. Then $| \mathcal{B}^n_m|=n-m-1$.
So there exists a path $\langle \textbf{u},\textit{P}_s,\textbf{v} \rangle$ between $\textbf{u}$ and $\textbf{v}$ in $Q_n-\mathcal{F}^n_m$, where $\textit{P}_s$ is in ${Q^0_n}-\mathcal{B}^n_m$. By the induction hypothesis, we have $l(\textit{P}_s) \le D({Q^0_n}-\mathcal{B}^n_m) \le D^{sc}_f(Q_{n-1};Q_{m-1}) \le n$.
Thus $d_{Q_n-\mathcal{F}^n_m}(\textbf{u}$, $\textbf{v})\le n$.

\noindent \textbf{Case 2.3.} $| \mathcal{A}^n_{m,0}|+| \mathcal{B}^n_m|\le n-m-2$. Then
there exists a path $\langle \textbf{u},\textit{P}_s, \textbf{v} \rangle$ between $\textbf{u}$ and $\textbf{v}$ in $Q_n-\mathcal{F}^n_m$, where $\textit{P}_s$ is in ${Q^0_n}-\mathcal{A}^n_{m,0}-\mathcal{B}^n_m$. By Case 1.2, we get $l(\textit{P}_s)\le n$. Therefore $d_{Q_n-\mathcal{F}^n_m}(\textbf{u}$, $\textbf{v})\le n$.
\end{proof}

\begin{mylemma}\label{lemma3.24}
	For $m\ge 0$ and $n\ge m+3$, $D_f(Q_n;Q_m)\ge n+1$ and $D^{sc}_f(Q_n;Q_m)\ge n+1$.
\end{mylemma}
\begin{proof} We take $\textbf{x}=00\cdots0\in Q_m$ and assume $Q_m\subseteq {Q^0_n}$. Note that ${Q^0_n}=Q_m\Box Q_{n-m-1}$. Let $\textbf{t}=\textbf{x}$, and we take $N_{Q_{n-m-1}}(\textbf{t})=\{\textbf{t}_1, \textbf{t}_2,\ldots, \textbf{t}_{n-m-1}\}$. We set $\textbf{y}=11\cdots10\in {Q^0_n}$. Clearly, $\bigcup^{n-m-1}_{i=1}(Q_m,\textbf{t}_i)$ contains $n-m-1$ subcubes and the removal of $\bigcup^{n-m-1}_{i=1}(Q_m,\textbf{t}_i)$ disconnects $Q^0_n$ such that $Q_m$ and $\textbf{y}$ are in the distinct components. For any $\textbf{z}\in Q_m$, we have that $\textbf{z}$ and $\textbf{y}$ are also in the distinct components of ${Q^0_n}-\bigcup^{n-m-1}_{i=1}(Q_m,\textbf{t}_i)$. Since ${Q_n}-\bigcup^{n-m-1}_{i=1}(Q_m,\textbf{t}_i)$ is connected, there is a path $\langle \textbf{x}, P_1, \textbf{z}, (\textbf{z})^{n}, \textit{P}_2, \textbf{y} \rangle$ between $\textbf{x}$ and $\textbf{y}$ in ${Q_n}-\bigcup^{n-m-1}_{i=1}(Q_m,\textbf{t}_i)$. Note that $H_{Q_n}((\textbf{z})^{n},\textbf{y})=n-s$ if $H_{Q_m}(\textbf{x},\textbf{z})=s$,  and $0\le s\le m$. Then $l(P_1)\ge s$ and $l(P_2)\ge n-s$.
Thus, the length of any path between $\textbf{x}$ and $\textbf{y}$ is at least $n+1$. This implies $D_f(Q_n;Q_m)\ge n+1$. Since $D_f(Q_n;Q_m)\le D^{sc}_f(Q_n;Q_m)$, $D^{sc}_f(Q_n;Q_m)\ge n+1$.
\end{proof}

By Lemmas~\ref{lemma3.23} and \ref{lemma3.24}, we have the following theorem:

\begin{mytheorem}\label{theorem3.25}
	$D^{sc}_f(Q_n;Q_m)=n+1$ for $m\ge 0$ and $n\ge m+3$.
\end{mytheorem}

On one hand, according to Theorem~\ref{theorem3.25} and the fact $D_f(Q_n;Q_m)\le D^{sc}_f(Q_n;Q_m)$, we obtain 
\linebreak
$D_f(Q_n;Q_m)\le n+1$. On the other hand, by Lemma~\ref{lemma3.24}, we derive $D_f(Q_n;Q_m)\ge n+1$. Thus, $D_f(Q_n;Q_m)= n+1$. Combining this with Theorem~\ref{theorem3.20}, we get the following theorem:

\begin{mytheorem}\label{theorem3.26}
	For $n \geq 3$ and $0 \leq m \leq n - 2$, $D_f(Q_n;Q_m)=$
	$\begin{cases}
		n  &  \textit{if } n=m+2,\\
		n+1 & \textit{if } n\ge m+3.
	\end{cases}$
\end{mytheorem}

\section{Concluding remarks}

In this paper, we introduced two novel concepts of fault diameters for graphs: the \textit{structure fault diameter} and the \textit{substructure fault diameter}, applying these definitions to analyze the $n$-dimensional hypercube $Q_n$. Unlike the conventional wide diameter approach, our method focuses on determining the maximum diameter between any two vertices under structure faults.
A natural extension of this work would involve determining the $K_{1,n}$-structure fault diameter and $K_{1,n}$-substructure fault diameter of various graph structures, further enriching our understanding of network robustness in the presence of complex fault scenarios.


\end{document}